\documentclass[12pt]{amsart}

\usepackage{amsmath}
\usepackage{amssymb}  
\usepackage{latexsym} 
\usepackage{comment}
\usepackage{url}
\usepackage{fullpage,url,amssymb,amsmath,amsthm,amsfonts,mathrsfs}
\usepackage[usenames,dvipsnames]{color}
\usepackage[pagebackref = true, colorlinks = true, linkcolor = blue, citecolor = Green]{hyperref}
\usepackage[alphabetic,lite,nobysame]{amsrefs}
\usepackage{enumitem}
\usepackage{amscd}   
\usepackage[all, cmtip]{xy} 
\usepackage{xfrac}

\DeclareFontFamily{U}{wncy}{}
\DeclareFontShape{U}{wncy}{m}{n}{<->wncyr10}{}
\DeclareSymbolFont{mcy}{U}{wncy}{m}{n}
\DeclareMathSymbol{\Sh}{\mathord}{mcy}{"58} 
\usepackage{color} 

\newcommand{\defi}[1]{\textsf{#1}} 

\def\act#1#2%
  {\mathop{}%
   \mathopen{\vphantom{#2}}^{#1}%
   \kern-3\scriptspace%
   #2}

\newcommand{\A}{{\mathbb A}}

\newcommand{\Xbar}{{\overline{X}}}

\newcommand{\calB}{{\mathcal B}}

\newcommand{\fdesc}{\textup{f-desc}}

\DeclareMathOperator{\Br}{Br}

\DeclareMathOperator{\HH}{H}

\DeclareMathOperator{\Spec}{Spec}

\DeclareMathOperator{\Sec}{Sec}

\newcommand{\HPdesc}{\operatorname{HP}^\fdesc}
\newcommand{\WAdesc}{\operatorname{WA}^\fdesc}
\newcommand{\HPmw}{\operatorname{HP}^{\textup{MW}}}
\newcommand{\WAmw}{\operatorname{WA}^{\textup{MW}}}
\newcommand{\Sel}{\operatorname{Sel}}

\newcommand{\res}{\operatorname{res}}

\newtheorem{Theorem}{Theorem}[section]
\newtheorem{Lemma}[Theorem]{Lemma}
\newtheorem{Proposition}[Theorem]{Proposition}
\newtheorem{Corollary}[Theorem]{Corollary}

\numberwithin{equation}{section}

\begin{document}

\title{Weak approximation versus the Hasse principle for subvarieties of abelian varieties}

\author{Brendan Creutz}
\address{School of Mathematics and Statistics, University of Canterbury, Private Bag 4800, Christchurch 8140, New Zealand}
\email{brendan.creutz@canterbury.ac.nz}
\urladdr{http://www.math.canterbury.ac.nz/\~{}bcreutz}

\begin{abstract}
	For varieties over global fields, weak approximation in the space of adelic points can fail. For a subvariety of an abelian variety one expects this failure is always explained by a finite  descent obstruction, in the sense that the rational points should be dense in the set of (modified) adelic points surviving finite descent. We show that this follows from the a priori weaker assumption that finite descent is the only obstruction to the existence of rational points. We also prove a similar statement for the obstruction coming from the Mordell-Weil sieve.
\end{abstract}

\maketitle

\section{Introduction}

Let $X$ be a smooth, projective, geometrically irreducible variety over a global field $k$. The set $X(\A_k)_{\bullet}^{\fdesc}$ of (modified) adelic points surviving descent by all $X$-torsors under finite group schemes is a closed subset of $X(\A_k)_{\bullet}$ containing the set of rational points $X(k)$. Here $X(\A_k)_{\bullet}$ is the set of connected components of the adelic points $X(\A_k) = \prod_vX(k_v)$, endowed with the quotient topology.

Consider the following statements, in which $\overline{X(k)}$ is used to denote the topological closure of $X(k)$ in $X(\A_k)_\bullet$.
\begin{align*}\label{HPdesc}
	 \HPdesc: & \quad\quad\quad \text{If $X(k) = \emptyset$, then $X(\A_k)_{\bullet}^{\fdesc} = \emptyset$.}\\
	 \WAdesc: & \quad\quad\quad X(\A_k)_{\bullet}^{\fdesc} = \overline{X(k)}
\end{align*}
 One says that the finite descent obstruction is the only obstruction to the Hasse principle (resp. to weak approximation) when $\HPdesc$ holds (resp. when $\WAdesc$ holds). 
 
 For $X \subset A$ a subvariety of an abelian variety we also consider analogous statements for the Mordell-Weil sieve, $X(\A_k)_\bullet \cap \overline{A(k)} \subset A(\A_k)_{\bullet}$.
		\begin{align*} 
			\HPmw: \quad\quad \quad & \text{If $X(k) = \emptyset$, then $X(\A_k)_{\bullet}\cap\overline{A(k)} = \emptyset$}\,. \\
			\WAmw: \quad\quad \quad & X(\A_k)_{\bullet} \cap \overline{A(k)} = \overline{X(k)}
		\end{align*}

For a given variety, or class of varieties, it is clear that $\WAdesc$ implies $\HPdesc$ and similarly for $\WAmw$ and $\HPmw$. The following theorems give converse results for subvarieties of abelian varieties. More precise results are proved in Corollary~\ref{thm:mainthm} and Theorem~\ref{thm:MW} below.

\begin{Theorem}\label{thm:mainthm_nodetail}
	If $\HPdesc$ holds for all subvarieties of abelian varieties, then $\WAdesc$ holds for all subvarieties of abelian varieties.
\end{Theorem}

\begin{Theorem}\label{thm:MWintro}
	Let $A/k$ be an abelian variety over a global field $k$. If $\HPmw$ holds for all closed subvarieties of $A$, then $\WAmw$ holds for all closed subvarieties of $A$.
\end{Theorem}

For a subvariety $X \subset A$ of an abelian variety, the sets appearing in these statements and the set of adelic points cut out by the Brauer-Manin obstruction are related as follows.
\begin{equation}\label{eq:sets}
	\overline{X(k)} \subseteq X(\A_k)_{\bullet}^{\fdesc} \subseteq X(\A_k)_\bullet^{\Br} \subseteq X(\A_k)_{\bullet} \cap A(\A_k)_{\bullet}^{\Br} \stackrel{\star}\supseteq X(\A_k)_{\bullet} \cap \overline{A(k)} \supseteq \overline{X(k)}
\end{equation}

It is conjectured that all of the sets appearing in~\eqref{eq:sets} are equal (See \cite[p. 133]{Skorobogatov} and \cite[Section 8]{Stoll} in the case of a curve inside its Jacobian, and \cite[7.4]{Poonen} and in \cite[Conjecture C]{PoonenVoloch} in general). The containment $\star$ is known to be an equality if the Tate-Shafarevich group of $A$ has no nontrivial divisible elements \cite{Scharaschkin}. The equality of all other sets is known to hold if in addition $A(k)$ is finite \cite{Scharaschkin,Stoll}. In the function field case, equality of all sets in~\eqref{eq:sets} is known for coset free subvarieties of abelian varieties which have no isotrivial isogeny factors and finite separable $p$-torsion \cite{PoonenVoloch} and for nonisotrivial curves of genus $\ge 2$ \cite{CreutzVoloch}.		

The weaker conjectures that $\HPmw$ and $\HPdesc$ hold for subvarieties of abelian varieties are supported by a wealth of empirical evidence as well as by a heuristic of Poonen \cite{Poonen}. (The results there are stated for a curve in its Jacobian, but the argument only relies on the fact that $X$ has positive codimension in $A$.) Theorem~\ref{thm:MWintro} shows that this same heuristic supports the stronger conjecture $\WAmw$ as well.

	A statement similar to Theorem~\ref{thm:mainthm_nodetail} has been observed for curves over number fields by Stoll \cite[Section 8]{Stoll} in which case it is also closely related to a well known result concerning the section conjecture in anabelian geometry which we recall in Section~\ref{sec:sec} which has its origins in work of Nakamura and Tamagawa. Those arguments rely on Faltings' Theorem that $X(k)$ is finite for curves of genus at least $2$ over number fields and so do not immediately generalize. Our proof relies instead on Theorem~\ref{thm:topology} below which provides a description of the toplogy on $X(\A_k)_{\bullet}^{\fdesc}$ in terms of torsors over $X$. This also provides further insight into the connection between $\WAdesc$ and the section conjecture, specifically Proposition~\ref{prop:sec}.

\section{The topology on $X(\A_k)_{\bullet}^{\fdesc}$}\label{sec:topology}

In this section we will prove the following theorem. 

\begin{Theorem}\label{thm:topology}
	Suppose $X \subset T$ is a subvariety of a torsor under an abelian variety over a global field $k$. Let $\calB$ be the collection of subsets of $X(\A_k)_{\bullet}^{\fdesc}$ of the form $U_f = f(X'(\A_k)_{\bullet}^{\fdesc})$, where $f : X' \to X$ is the pullback of a geometrically connected torsor $T' \to T$ under a finite group scheme over $k$. Then $\calB$ is a basis for the subspace topology on $X(\A_k)_{\bullet}^{\fdesc} \subset T(\A_k)_\bullet $.
\end{Theorem}
	
\begin{Corollary}\label{thm:mainthm}
	The following are equivalent for $X \subset T$ a subvariety of a torsor under an abelian variety over a global field $k$.
	\begin{enumerate}
		\item\label{itWA} $\WAdesc$ holds for $X$;
		\item\label{itHP} $\HPdesc$ holds for every $X'/k$ which arises as the pullback of a geometrically connected torsor $T' \to T$ under a finite group scheme over $k$;
		\item\label{itWAall} $\WAdesc$ holds for every $X'/k$ which arises as the pullback of a geometrically connected torsor $T' \to T$ under a finite group scheme over $k$.
	\end{enumerate}
\end{Corollary}

\begin{proof}[Proof of Corollary~\ref{thm:mainthm}]
	We first show that~\eqref{itHP} implies~\eqref{itWA}. Let $x \in X(\A_k)_{\bullet}^{\fdesc}$. It suffices to show that every open subset $U \subset X(\A_k)_{\bullet}^{\fdesc}$ which contains $x$ also contains a $k$-point. By Theorem~\ref{thm:topology}, there is an open subset of $U$ of the from $U_f = f(X'(\A_k)_{\bullet}^{\fdesc}) $ containing $x$ with $f : X' \to X$ the pullback of a geometrically connected torsor $T' \to T$. If we assume $X'$ satisfies $\HPdesc$, then we have $X'(k) \ne \emptyset$ and so $U$ contains a rational point.
	
	To see that~\eqref{itWA} implies~\eqref{itHP}, suppose there is some $X' \to X$ which does not satisfy $\HPdesc$. Then $X'(\A_k)_{\bullet}^{\fdesc}$ is nonempty, but $X'(k) = \emptyset$. By Theorem~\ref{thm:topology} we have that $f(X'(\A_k)_{\bullet}^{\fdesc})$ is a nonempty open subset of $X(\A_k)_{\bullet}^{\fdesc}$ containing no $k$-rational point. So $\overline{X(k)} \ne X(\A_k)_{\bullet}^{\fdesc}$.
	
	Clearly~\eqref{itWAall} implies~\eqref{itWA}, so to complete the proof it now suffices to show that~\eqref{itHP} implies~\eqref{itWAall}. Let $X' \to X$ be the pullback of $T' \to T$. Every torsor $X'' \to X'$ obtained by pulling back a geometrically connected torsor $T'' \to T'$ can be composed to give a torsor over $X$ arising as pullback from a geometrically connected torsor over $T$. Thus, if~\eqref{itHP} holds for $X \subset T$, then it must also hold for $X' \subset T'$ and we can conclude using the impplication~\eqref{itHP} implies~\eqref{itWA} for $X' \subset T'$.
	\end{proof}


\begin{Lemma}\label{lem:topA}
	Let $A$ be an abelian variety over a global field $k$. For every open subset $U \subset A(\A_k)_{\bullet}$ containing the identity there exists an integer $n$ such that $nA(\A_k)_{\bullet} \subset U$. 
\end{Lemma}

\begin{proof}
	Any open set in $A(\A_k)_\bullet$ is of the form $U = \prod_{v \in S} U_v \times \prod_{v \not\in S} A(k_v)_{\bullet}$ with $U_v \subset A(k_v)_\bullet$ open and $S$ finite. So it suffices to show that every open subset $U_v \subset A(k_v)_\bullet$ contains a set of the form $nA(k_v)$ for some $n$. For archimedean $v$ this is clear since $A(k_v)_\bullet$ is a finite group. For nonarchimedean $v$, 	this follows from the fact that $A(k_v)$ contains a finite index torsion free pro-$p$-subgroup (where $p$ is the residue characteristic of $k_v$). In the number field case this is Mattuck's theorem \cite{Mattuck} and in general follows from properties of the formal group of $A(k_v)$ (e.g., \cite[pp. 116-118]{Serre}).
\end{proof}

\begin{Lemma}\label{lem:2}
	Let $X \subset T$ be a subvariety of a torsor under an abelian variety over a global field $k$. The topology on $X(\A_k)_{\bullet}^{\fdesc}$ generated by the subsets $U_f := f(X'(\A_k)_\bullet^{\fdesc})$ with $f : X' \to X$ the pullback of a geometrically connected torsor $T' \to T$ under a finite abelian group scheme is at least as strong as the subspace topology $X(\A_k)_{\bullet}^{\fdesc} \subset X(\A_k)_{\bullet}$.
\end{Lemma}

\begin{proof}
	Suppose $x \in X(\A_k)_{\bullet}^{\fdesc}$ and let $U \subset X(\A_k)_{\bullet}^{\fdesc}$ be an open subset containing $x$. It is enough to find $f$ such that $x \in U_f \subset U$. Then $U = V \cap X(\A_k)_{\bullet}^{\fdesc}$ for some open subset $V \subset T(\A_k)_{\bullet}$. 
	
	Let $A$ denote the abelian variety for which $T$ is a torsor. The torsor structure on $T$ gives a homeomorphism
	\[
		T(\A_k)_{\bullet} \ni y \mapsto y - x \in A(\A_k)_{\bullet}\,,
	\]
	sending $x$ to the identity and $V$ to a neighbourhood $W \subset A(\A_k)_{\bullet}$ of the identity. By Lemma~\ref{lem:topA} there is an integer $n$ such that $nA(\A_k)_{\bullet} \subset W$.
	
	Let $g : T' \to T$ be an $n$-covering (i.e., a twist of the multiplication by $n$ map on $A$) and let $f : X' \to X$ be the pullback to $X$. By \cite[Prop 5.17]{Stoll} we have
	\begin{equation*}\label{eq:fdesc=}
		X(\A_k)_\bullet^{\fdesc} = \bigcup_{\tau \in \Sel(f)} f^\tau(X'^\tau(\A_k)_\bullet^{\fdesc})\,,
	\end{equation*}
	the union ranging over the twists of $f$ for which $X'^\tau$ has adelic points. The result is stated there for $k$ a number field, but the proof also works for global function fields noting that $\Sel(f)$ is also finite in this case by \cite{Milne}. Therefore, by replacing $X'$ with a twist if needed, we can assume $x \in f(X'(\A_k)_\bullet^\fdesc)$.
	
	Note that $f(X'(\A_k)_\bullet) \subset g(T'(\A_k)_{\bullet}) \subset V$ since the image of $g(T'(\A_k)_{\bullet})$ in $A(\A_k)_{\bullet}$ is equal to $nA(\A_k)_\bullet$ which is contained in $W$ by our assumption on $n$. So $x \in U_f := f(X'(\A_k)_{\bullet}^{\fdesc})  \subset U = X(\A_k)_{\bullet}^{\fdesc} \cap V$ as required.
\end{proof}

\begin{Lemma}\label{lem:3}
	Let $X \subset T$ be a subvariety of a torsor under an abelian variety over a global field $k$ and let $Y \to X$ be the pullback of a geometrically connected torsor $T' \to T$ under a finite abelian group scheme $G/k$. Then the set $U_f := f(Y(\A_k)_\bullet^{\fdesc})$ is an open subset of $X(\A_k)_{\bullet}^{\fdesc}$.
\end{Lemma}

\begin{proof}
	We must show that $U_f = U \cap X(\A_k)_{\bullet}^{\fdesc}$ for some open subset $U \subset X(\A_k)$. By \cite[Prop. 5.17]{Stoll} we have 
	$$
		X(\A_k)_{\bullet}^{\fdesc} = \bigcup_{\tau \in \Sel(f)} f^\tau(Y^\tau(\A_k)_{\bullet}^{\fdesc})\,,
	$$
	where $f^\tau : Y^\tau \to X$ denote the twists of $f$ ranging over the subset $\Sel(f) \subset \HH^1(k,G)$ parameterizing twists of $Y$ that have adelic points. Note that $\Sel(f)$ is finite in the \'etale case by \cite[8.4.6]{Qpoints} and by \cite{Milne} in in general.
	
	Since $\Sel(f)$ is finite, there is a finite subset $S \subset \Omega_k$ of places of $k$ such that for any $\tau,\tau' \in \Sel(f)$ we have that $\res_v(\tau)$ and $\res_v(\tau')$ differ at some prime $v \in S$ or they agree for all places $v \in \Omega_k$. For each $v \in S$, the subsets $f^\tau(Y^\tau(k_v)) \subset X(k_v)$ are closed, being the continuous image of a compact set in a Hausdorff space. Moreover, for varying $\tau$ these sets are either pairwise equal or disjoint, so $$W_v := \left( \bigcup_{\tau \in \Sel(f)} f^\tau(Y^\tau(k_v)) \right) \setminus f(Y(k_v))$$ is a closed subset of $X(k_v)$ whose complement $W_v^c$ contains $f(Y(k_v))$ and is disjoint from $f^\tau(Y^\tau(k_v))$ for any $\tau \in \Sel(f)$ such that $\res_v(\tau) \ne 0$. Let $U := \prod_{v \in S} W_v^c \times \prod_{v \not\in S} X(k_v)$. This is an open subset of $X(\A_k)$ with the property that $U \cap X(\A_k)_{\bullet}^{\fdesc} = f(Y(\A_k)_{\bullet}^{\fdesc}) = U_f$.
\end{proof}

In the number field case we have the following strengthening which will be used in the following section.

\begin{Lemma}\label{lem:str}
	Let $f : X' \to X$ be an \'etale morphism. Then $f(X'(\A_k)_{\bullet}^\fdesc) \subset X(\A_k)_\bullet^\fdesc$ is open.
\end{Lemma}

\begin{proof}
	The proof of the preceding lemma shows that $f(X'(\A_k)_{\bullet}^\fdesc) \subset X(\A_k)_\bullet^\fdesc$ is open for any torsor $f:X'\to X$ under a finite group scheme over $k$ such that $\Sel(f)$ is finite, which always holds in the number field case. If $f : X' \to X$ is merely \'etale, then we use that there is an \'etale torsor $g : X'' \to X$ which factors as the composition of $f$ with an \'etale torsor $h : X'' \to X'$ \cite[Proposition 5.3.9]{Szamuely}. By the previous lemma the sets $g^\tau(X''^\tau(\A_k)_{\bullet}^\fdesc) \subset X(\A_k)_{\bullet}^\fdesc$ are open for any twist of $g$. Thus by \cite[Prop. 5.17]{Stoll} applied to $h : X'' \to X'$, we see that $f(X'(\A_k)^\fdesc_\bullet) = \bigcup_{\tau \in \Sel(h)} (f\circ h^\tau)(X''^\tau(\A_k)_{\bullet}^\fdesc)$ is a union of open sets and hence open.
\end{proof}

\begin{proof}[Proof of Theorem~\ref{thm:topology}]
	This follows immediately from Lemma~\ref{lem:2} and Lemma~\ref{lem:3}.
\end{proof}

\section{The Section Conjecture}\label{sec:sec}

Suppose $X$ is a smooth, proper, geometrically irreducible variety over a number field $k$. By functoriality, any $k$-rational point determines a section of the fundamental exact sequence,
\begin{equation}\label{pi1}
	1 \to \pi_1(\Xbar) \to \pi_1(X) \to \pi_1(\Spec(k)) \to 1\,,
\end{equation}
well defined up to conjugation by an element of $\pi_1(\Xbar)$. Grothendieck's section conjecture asserts that when $X$ is a hyperbolic curve the map $X(k) \to \Sec(X/K)$ is a bijection. It is known that the full section conjecture for $X$ follows if one assumes a weaker form of the section conjecture for all geometrically connected \'etale coverings $X' \to X$, namely that the existence of a section in $\Sec(X'/k)$ implies the existence of a $k$-rational point on $X'$ \cite[Theorem 54 and Section 9.4]{Stix}. Corollary~\ref{thm:mainthm} is an analogue of this result for adelic points surviving descent. 

 A section is called \defi{Selmer} if for every place $v$ of $k$ its restriction to the decomposition group at $v$ arises from a $k_v$-point. When $X \subset A$ is a subvariety of an abelian variety such a $k_v$-point (if it exists) is unique (up to the usual caveat at archimedean places) by \cite[Proposition 73]{Stix}. This defines a map $\operatorname{loc}:\Sec^{\Sel}(X/k) \to X(\A_k)_\bullet^{\fdesc}$ from the set of Selmer sections (considered up to $\pi_1(\Xbar)$-conjugacy) to adelic points surviving descent. This map is surjective by \cite[Theorem 2.1]{HarariStix} (See also \cite[Theorem 144]{Stix} and \cite[Remark 8.9]{Stoll}). Surjectivity of $\operatorname{loc}$ shows that the section conjecture for $X$ implies $\WAdesc$ for $X$. It is shown in \cite[Proposition 2.13]{Bettsetal} that $\WAdesc$ is equivalent to the Selmer section conjecture, which asserts that $X(k) \to \Sec^{\Sel}(X/k)$ is a bijection.

The space $\Sec(X/k)$ of sections up to conjugacy is naturally equipped with a prodiscrete topology, inducing a subspace topology on $\Sec^{\Sel}(X/k)$. Theorem~\ref{thm:topology} implies the following.

\begin{Proposition}\label{prop:sec}
	Suppose $X \subset A$ is a closed subvariety of an abelian variety over a number field $k$. The surjective map $\operatorname{loc}:\Sec^{\Sel}(X/k) \to X(\A_k)_\bullet^{\fdesc}$ is continuous and open.
\end{Proposition}

\begin{proof}
	The sets 
	\[
		U_{X'} := \textup{image}\left(\Sec(X'/k) \to \Sec(X/k)\right)
	\]
	as $X' \to X$ ranges over geometrically connected \'etale coverings of $X$ form a basis for the topology on $\Sec(X/k)$ \cite[Section 4.2]{Stix}. The sets $V_{X'} := U_{X'} \cap \Sec^{\Sel}(X/k)$ thus form a basis for the subspace topology on $\Sec^{\Sel}(X/k) \subset \Sec(X/k)$.  By \cite[Lemma 2.15]{Bettsetal} (which is stated for an \'etale covering of hyperbolic curves, but whose proof works for any \'etale covering of geometrically connected varieties) we have that $V_{X'} = \textup{image}\left(\Sec^{\Sel}(X'/k) \to \Sec^{\Sel}(X/k)\right)$. Thus for any \'etale $f : X' \to X$ we have
	\begin{enumerate}
		\item\label{it1} $\operatorname{loc}(V_{X'}) = f(X'(\A_k)^{\fdesc}_\bullet)$, and
		\item\label{it2} $\operatorname{loc}^{-1}(f(X'(\A_k)^{\fdesc}_\bullet)) = V_{X'}$.
	\end{enumerate}
	Lemma~\ref{lem:str} together with~\eqref{it1} shows that $\operatorname{loc}$ is open. Lemma~\ref{lem:2} together with~\eqref{it2} shows that $\operatorname{loc}$ is continuous.
\end{proof}

\section{The Mordell-Weil Sieve}

\begin{Theorem}\label{thm:MW}
	Let $X \subset A$ be a closed subvariety of an abelian variety over a global field $k$. If $\HPmw$ holds for every $Y \subset A$ which is the pullback to $X$ of a morphism $\rho : A \to A$ of the form $\rho(a) = na + P$ with $n \ge 1$ and $P \in A(k)$, then $\WAmw$ holds for $X$.
\end{Theorem}

\begin{proof}
	Let $X \subset A$ be a closed subvariety and suppose $x \in X(\A_k)_\bullet \cap \overline{A(k)}$.  Let $U \subset X(\A_k)_\bullet$ be an open subset containing $x$. For the first statement it suffices to show that $U \cap X(k)$ is nonempty. Let $V \subset A(\A_k)_\bullet$ be an open subset such that $U = V \cap X(\A_k)_\bullet$. By Lemma~\ref{lem:topA} there is an integer $n_0$ such that $(x + n_0A(\A_k)_{\bullet}) \subset V$
	
	Since $x \in \overline{A(k)}$ there exists a sequence of points $P_n \in A(k)$ such that
	\begin{enumerate}
		\item\label{it:1} $x \in P_n + nA(\A_k)_{\bullet}$ for all $n \ge 1$, and
		\item\label{it:2} $P_{mn} \equiv P_n \bmod nA(k)$, for all $m,n \ge 1$.
	\end{enumerate}
	The existence of $P_n$ satisfying~\eqref{it:1} follows from Lemma~\ref{lem:topA}. That these $P_n$ can be chosen to also satisfy~\eqref{it:2} follows from K\"onig's Lemma. This same argument shows as in \cite[Theorem E]{PoonenVoloch} that the inclusion $A(k) \to A(\A_k)_{\bullet}$ induces an isomorphism $\widehat{A(k)} \simeq \overline{A(k)}$ between the profinite completion of $A(k)$ and its topological closure in $A(\A_k)_{\bullet}$.
	
	For $m \ge 1$, let $R_{m} \in A(k)$ be such that $P_{mn_0} - P_{n_0} = n_0R_{m}$. The possible choices for $R_{m}$ differ by an element in $A(k)[n]$. Since $A(k)_{\textup{tors}}$ is finite we can choose the $R_{m}$ so that for any $\ell,m \ge 1$ we have $R_{\ell m} \equiv R_{m} \bmod mA(k)$. With such a choice, the limit $y := \lim_{m \to \infty} R_{m!} \in \overline{A(k)}$ exists. Let $\rho : A \to A$ be the morphism given by $\rho(a) = n_0a + P_{n_0}$. Since $\rho : A(\A_k)_\bullet \to A(\A_k)_\bullet$ is continuous, we have $\rho(y) = \rho(\lim_{m} R_{m!}) = \lim_{m} n_0R_{m!} + P_{n_0} = \lim_{m} P_{n_0m!} = x$.
	
	Let $Y \to X$ be the pullback of $\rho$ to $X$. Then $Y \subset A$ and from above we have that $y \in Y(\A_k)_{\bullet} \cap \overline{A(k)}$. Assuming $\HPmw$ holds for $Y$ we find $Y(k) \ne \emptyset$. Since $\rho(Y(\A_k)_\bullet) \subset U$, this shows that $U\cap X(k)\ne \emptyset$.
\end{proof}

Under the additional assumption that $X(k)$ is finite we have the converse to Theorem~\ref{thm:MW}. 

\begin{Proposition}
	Suppose $X \subset A$ is a subvariety of an abelian variety over a global field with $X(k)$ finite. If $\WAmw$ holds for $X$, then $\WAmw$ holds for every $Y \subset A$ obtained as the pullback of a morphism $\rho : A \to A$ of the form $\rho(a) = na + b$ with $n \ge 1$ and $b \in A(k)$.
\end{Proposition}

\begin{proof}
	Let $y \in Y(\A_k)_\bullet \cap \overline{A(k)}$. Then $x = \rho(y) \in X(\A_k)_\bullet \cap \overline{A(k)} = \overline{X(k)} = X(k)$ since $X$ satisfies $\WAmw$ and $X(k)$ is finite. The zero-dimensional subscheme $Z := \rho^{-1}(x) \subset Y \subset A$ satisfies $\WAmw$ by \cite[Theorem 3.11]{Stoll} in the number field case and by \cite[Theorem E]{PoonenVoloch} and~\cite[Proposition 3.1]{CreutzVoloch} in the function field case. Since $y \in Z(\A_k)_\bullet \cap \overline{A(k)}$ we conclude that $y \in Y(k)$.
\end{proof}	


\section*{Acknowledgements}
This research was supported by the Marsden Fund Council, managed by Royal Society Te Ap\=arangi.

\end{document}